\documentclass{amsart}

\usepackage{amsfonts, amsmath, amscd}
\usepackage[psamsfonts]{amssymb}

\usepackage{amssymb}

\newtheorem{theorem}{Theorem}
\newtheorem{lemma}[theorem]{Lemma}

\theoremstyle{definition}

\theoremstyle{remark}

\numberwithin{equation}{section}

\begin{document}

\title{Reflexive group topologies on Abelian groups}

\author{S.S. Gabriyelyan }

\address{S.S. Gabriyelyan, Department of Mathematics, Ben-Gurion University of the
Negev, Beer-Sheva, P.O. 653, Israel}
\email{saak@math.bgu.ac.il}

\thanks{The author was partially supported
 by Israel Ministry of Immigrant Absorption.}


\date{may 2009}


\begin{abstract}
It is proved that any infinite Abelian group  of infinite exponent admits a non-discrete reflexive group topology.
\end{abstract}

\maketitle

\section*{Introduction}

For a topological group $G$, the group $G^{\wedge}$ of continuous homomorphisms (characters) into the torus $\mathbb{T} =\{ z\in \mathbb{C} : \; |z|=1\}$ endowed with the compact-open topology is called  the {\it character group} of $G$ and $G$ is named {\it Pontryagin reflexive} or {\it reflexive} if the canonical homomorphism $\alpha_G : G\to G^{\wedge\wedge} , g\mapsto (\chi\mapsto (\chi, g))$ is a topological isomorphism. In the article we consider the following question.

\[
\begin{split}
\mbox{ {\bf Problem 1.}} &  \mbox{ {\it Is any infinite Abelian group admits a non-discrete reflexive }} \\
  & \mbox{ {\it group topology}?}
\end{split}
\]

A group $G$ with the discrete topology is denoted by $G_d$. The exponent of $G$ (=the least common multiple of the orders of the elements of $G$) is denoted by $\exp G$. The subgroup of $G$ which generated by an element $g$ is denoted by $\langle g\rangle$.

Following E.G.Zelenyuk and I.V.Protasov \cite{ZP1}, we say that a sequence $\mathbf{u} =\{ u_n \}$ in a group $G$ is a $T$-{\it sequence} if there is a Hausdorff group topology on $G$ for which $u_n $ converges to zero. The group $G$ equipped with the finest group topology with this property is denoted by $(G, \mathbf{u})$. Using the method of $T$-sequences, they proved that every infinite Abelian group admits a complete group topology for which characters do not separate points. Using this method,
we give the positive answer to Problem 1 for groups of infinite exponent. We prove the following theorem.

\begin{theorem} \label{t1}
{\it Any infinite Abelian group $G$ such that $\exp G=\infty$ admits a non-discrete  reflexive group topology.}
\end{theorem}

Let $G$ be a Borel subgroup of a Polish group $X$. $G$ is called polishable if there exists a Polish group topology $\tau$ on $G$ such that the identity map $i : (G, \tau ) \to X, i(g)=g,$ is continuous.
A $\delta$-neighborhood of zero in a Polish group is denoted by $U_\delta$.

Let $X$ be a compact metrizable group and $\mathbf{u} =\{ u_n \}$ a sequence of elements of $X^\wedge$. We denote by $s_{\mathbf{u}} (X)$ the set of all $x\in X$ such that $(u_n , x)\to 1$. Let $G$ be a subgroup of $X$. If $G=s_{\mathbf{u}} (X)$ we say that $\mathbf{u}$ {\it characterizes} $G$ and that $G$ is {\it characterized} (by $\mathbf{u}$). By Theorem 1 \cite{Ga1}, if $G$ is characterized, then it is polishable by the following metric
\begin{equation} \label{01}
\rho (x,y) = d(x,y) + \sup \{ |(u_n, x) -(u_n, y)|, \; n\in \mathbb{N} \},
\end{equation}
where $d$ is the initial metric on $X$. $G$ with the metric $\rho$ is denoted by $G_\mathbf{u}$.

The integral part of a real number $x$ is denoted by $[x]$. By $\| x\|$ we denote the distance of a real number $x$ to the nearest integer. We also use the following inequality $\pi |\varphi | \leqslant | 1- e^{2\pi i\varphi} |\leqslant  2\pi |\varphi | , \varphi\in [-\frac{1}{2} ; \frac{1}{2} )$.

\section{The Proof}

Let $G$ be an infinite Abelian group and $H$ be its infinite subgroup. If $H$ admits a non-trivial reflexive group topology $\tau$, then we can extend $\tau$ to  $G$ such that $H$ will be an open subgroup. Then, by \cite{BCM}, $G$ also will be reflexive. It is well known that any Abelian group $G$ of infinite exponent contains at least one of the following groups:
\begin{enumerate}
\item $\mathbb{Z}$;
\item $\mathbb{Z} (p^{\infty})$ for some prime number $p$;
\item $G=\bigoplus_n \mathbb{Z} (b_n)$, where $b_1 \leqslant b_2 \leqslant\dots ,\; b_n \to\infty$.
\end{enumerate}
Thus it is enough to prove Theorem \ref{t1} for these three cases only.

\subsection{The case $\mathbb{Z}$}

A non-trivial reflexive group topology on $\mathbb{Z}$ is constructed in Theorem 2 \cite{Gab}.

\subsection{The case $\mathbb{Z} (p^{\infty})$}

\begin{proof}  Set
\[
\mathbf{u}=\{ u_k\},\; u_k =\frac{1}{p^{n_k +1}}, \mbox{ where } n_1<n_2<\dots \mbox{ and } n_{k+1} -n_k \to\infty.
\]
By (25.2) \cite{HR}, if $\omega=(a_n) \in \Delta_p =\mathbb{Z} (p^{\infty})^\wedge$, where $0\leqslant a_n <p$, then
\begin{equation} \label{1}
(u_k,\omega) =\exp \left\{ \frac{2\pi i}{p^{n_k +1}} \left( a_0 + pa_1 +\dots + p^{n_k} a_{n_k} \right)\right\} .
\end{equation}
By (10.4) \cite{HR}, if $\omega_1 \not= \omega_2 \in\Delta_p$, then
$d(\omega_1 ,\omega_2) = 2^{-n}$, where  $n$ is the minimal index such that $a_n^1 \not= a_n^2$.
For any $\omega=(a_n) \in \Delta_p$ and $k>1$, we put
\[
m_k =m_k (\omega) =\max \{ d_k, n_{k-1} \}, \mbox{ where } d_k  = n_k \mbox{ if } 0< a_{n_k} <p-1,
\]
\[
 \begin{split}
\mbox{ or } d_k =\min \{ j  : & \mbox{ either } a_s =0 \mbox{ for every } j<s\leqslant n_k, \\
& \mbox{or } a_s =p-1 \mbox{ for every } j<s\leqslant n_k \}.
\end{split}
\]
Then $n_{k-1}\leqslant m_k\leqslant n_k$ and the equality $m_k =n_k$ is possible only if $p\not= 2$.

Set $\omega_0 = (1,0,0,\dots)\in \Delta_p$. It is clear, by (\ref{1}), that $\omega_0 \in s_{\mathbf{u}} (\Delta_p )$. Since $\langle\omega_0 \rangle$ is dense in $\Delta_p$, then, by Theorem 3 \cite{Ga1}, $\mathbf{u}$ is a $T$-sequence and $(\mathbb{Z} (p^{\infty}), \mathbf{u})^\wedge = G_\mathbf{u}$, where the Polish group $G_\mathbf{u}$ is $s_{\mathbf{u}} (\Delta_p )$ with the Polish group metric $\rho$ [see (\ref{01})].

We need the following three lemmas.
\begin{lemma} \label{l11}
$\omega \in s_{\mathbf{u}} (\Delta_p ) \Leftrightarrow n_k -m_k \to\infty.$
\end{lemma}

\begin{proof} We can rewrite (\ref{1}) as follows:

(a) if $a_s =0$ for $m_k <s\leqslant n_k$, then $({\rm mod}\, 1)$
\begin{equation} \label{2}
\frac{1}{2\pi i} {\rm Arg} (u_k, \omega)= \frac{1}{p^{n_k -m_k +1}} \sum_{l=0}^{m_k} \frac{a_l}{p^{m_k -l}} ;
\end{equation}

(b) if $a_s =p-1$ for $m_k <s\leqslant n_k$, then $({\rm mod}\, 1)$
\[
\frac{1}{2\pi i} {\rm Arg} (u_k, \omega)= \frac{1}{p^{n_k -m_k +1}} \sum_{l=0}^{m_k} \frac{a_l}{p^{m_k -l}} + \sum_{l=m_k +1}^{n_k} \frac{p-1}{p^{n_k +1-l}} = ({\rm mod}\, 1)
\]
\begin{equation} \label{3}
= \frac{1}{p^{n_k -m_k +1}} \sum_{l=0}^{m_k} \frac{a_l}{p^{m_k -l}} -\frac{1}{p^{n_k -m_k }};
\end{equation}

(c) if $m_k =n_k$ (and, hence, $p\not= 2$),  then $({\rm mod}\, 1)$
\begin{equation} \label{4}
\frac{1}{2\pi i} {\rm Arg} (u_k, \omega)=  \sum_{l=0}^{n_k -1} \frac{a_l}{p^{n_k +1-l}} +\frac{a_{n_k}}{p}.
\end{equation}

Assume that $\omega\in s_{\mathbf{u}}(\Delta_p )$ and $n_k -m_k \not\to \infty$.

Let case (a) be fulfilled for $k_1<k_2< \dots$ and $n_{k_t} -m_{k_t} =r>0$. Then, by (\ref{2}), we have $({\rm mod}\, 1)$
\[
\frac{1}{p^{r+1}} \leqslant \frac{a_{m_{k_t}}}{p^{r+1}} \leqslant \frac{1}{2\pi i} {\rm Arg} (u_{k_t}, \omega) <\frac{1}{p^r},
\]
and, hence, $(u_{k_t}, \omega)  \not\to 1$. It is a contradiction.

Let case (b) be fulfilled for $k_1<k_2< \dots$ and $n_{k_t} -m_{k_t} =r>0$.
Since $a_{m_{k_t}} \leqslant p-2$, then
\[
\sum_{l=0}^{m_{k_t}} \frac{a_l}{p^{m_{k_t}-l}} < p-2 + \frac{1}{p} \sum_{l=0}^\infty \frac{p-1}{p^l} = p-2 +\frac{p-1}{p}\cdot \frac{p}{p-1}=p-1.
\]
Then, by (\ref{3}), we have $({\rm mod}\, 1)$
\[
-\frac{1}{p^{r}} \leqslant \frac{1}{2\pi i} {\rm Arg} (u_{k_t}, \omega) < \frac{1}{p^{r+1}} \cdot (p-1) -\frac{1}{p^r} = -\frac{1}{p^{r+1}},
\]
and, hence, $(u_{k_t}, \omega)  \not\to 1$. It is a contradiction.

Let case (c) be fulfilled for $k_1<k_2< \dots$ and $n_{k_t} =m_{k_t}$. Then $p>2$ and $0<a_{n_{k_t}} < p-1$. Thus, by (\ref{4}), we have $({\rm mod}\, 1)$
\[
\frac{1}{p} \leqslant\frac{a_{n_{k_t}}}{p} \leqslant \frac{1}{2\pi i} {\rm Arg} (u_{k_t}, \omega) <\frac{a_{n_{k_t}}}{p} + \frac{1}{p} \leqslant \frac{p-1}{p},
\]
and, hence, $(u_{k_t}, \omega)  \not\to 1$. It is a contradiction.

The converse assertion evidently follows from (\ref{2}) and (\ref{3}).
\end{proof}

\begin{lemma} \label{l12}
{\it   $\langle\omega_0 \rangle$ is dense in $G_{\mathbf{u}}$ and, hence, $G_\mathbf{u}$ is monothetic.}
\end{lemma}

\begin{proof}
Let $\varepsilon>0$ and $\omega =(a_n)\in G_{\mathbf{u}}$. Choose $r$ such that $\frac{1}{p^r} <\frac{\varepsilon}{10}$. By Lemma \ref{l11}, we can choose $k_0$ such that $\frac{1}{2^{n_{k_0 -1}}} < \frac{\varepsilon}{10}$ and $n_k -m_k >r+1$ for every $k\geqslant  k_0$. Choose $q$ such that
\[
\omega - \omega_0^q = (0,\dots , 0_{m_{k_0}}, a_{m_{k_0} +1}, \dots).
\]
Thus, if $k<k_0$, then, $(u_k,\omega - \omega_0^q)=1$. If $k\geqslant  k_0$, then, by (\ref{2}) and (\ref{3}),
\[
| 1-(u_k,\omega - \omega_0^q)|\leqslant 2\pi \cdot \frac{1}{p^{n_k -m_k}} < 0.7 \varepsilon.
\]
So $\rho (\omega , \omega_0^q)<\varepsilon$ and $\langle\omega_0 \rangle$ is dense in $G_{\mathbf{u}}$. \end{proof}

\begin{lemma} \label{l13}
{ $G^\wedge_\mathbf{u} = \mathbb{Z} (p^\infty)$ algebraically.}
\end{lemma}

\begin{proof} Since $\langle\omega_0 \rangle$ is dense in $G_\mathbf{u}$, any continuous character $\chi$ of $G_\mathbf{u}$ is defined by its value on $\omega_0$. Let $(\chi, \omega_0)=\exp \{ 2\pi i \alpha\}$ for some $\alpha \in [0,1)$. It is enough to prove that $\alpha\in \mathbb{Z} (p^\infty)$. Let
\[
\alpha = \sum_{i=1}^\infty \frac{b_i}{p^i}, \mbox{ where } 0\leqslant b_i <p.
\]

Let $0<\varepsilon<\frac{1}{p^2}$. Since $\chi$ is continuous, then there exists $\delta >0$ such that
\begin{equation} \label{5}
|1- (\chi, \omega)| < \varepsilon, \quad \forall \omega\in U_{\delta}.
\end{equation}
Choose $r_0$ and $k_0$ such that $\frac{1}{p^{r_0}} <\frac{\delta}{10}$ and $\frac{1}{2^{n_{k_0}}} <\frac{\delta}{10}$. Let $\omega=(a_n) \in s_{\mathbf{u}} (\Delta_p)$ have the following form
\begin{enumerate}
\item[(i)] There exist $k_0 <k_1<\dots <k_s , s\in \mathbb{N},$ such that $a_n\in [0,p-1]$ if $n\in [n_{k_i}+1, n_{k_i +1} -r_0-1], i=0,\dots,s-1,$ and $a_n =0$ otherwise.
\end{enumerate}
Then, by (\ref{01}), (\ref{1}) and (\ref{2}), we have
\[
\rho(0,\omega)\leqslant \frac{\delta}{10}+ \sup \{ |1-(u_n,\omega)|, \; n\in\mathbb{N} \} \leqslant \frac{\delta}{10} + \frac{2\pi}{p^{r_0}} <\delta,
\]
i.e. $\omega\in U_{\delta}$.

1) {\it Let} $p>2$. Set
\[
R(k) =\{ i: \; n_{k-1} +1<i\leqslant n_k +1 \mbox{ and } 0<b_i < p-1 \}.
\]
If $R(k)$ is not empty, we set $r_k =\min \{ i:\; i\in R(k)\}$.

{\it Let us prove that there exists $C_1 >0$ such that for every $k>k_0$ if $i\in R(k)$, then $0\leqslant n_k +1-i <C_1$. Hence there exists $k'_0 >k_0$ such that for any $k> k'_0$ if $n_{k-1} +1 <i\leqslant n_k +1-C_1$, then either $b_i = p-1$ or $b_i =0$.}

Assume the converse and there exists a subsequence $r_{k_q}$ such that $n_{k_q} -r_{k_q} \to\infty$.
We can assume that  $n_{k_q} -r_{k_q} >r_0$.  Since
\[
\omega_0^{p^{r_{k_q}-1}} =(0,\dots,0_{n_{k_q -1}},\dots, 0, 1_{r_{k_q} },0,\dots)
\]
 satisfies condition (i), then it is contained in $U_{\delta}$. On the other hand,
\[
\left( \chi, \omega_0^{p^{r_{k_q}-1}}\right) =\exp\left\{ 2\pi i \left( \frac{b_{r_{k_q}}}{p} +\sum_{i=1}^\infty \frac{b_{{r_{k_q}} +i}}{p^{i+1}} \right)\right\} \mbox{ and }
\]
\[
\frac{1}{p}\leqslant \frac{b_{r_{k_q}}}{p} \leqslant \frac{b_{r_{k_q}}}{p} +\sum_{i=1}^\infty \frac{b_{{r_{k_q}} +i}}{p^{i+1}} \leqslant \frac{b_{r_{k_q}}+1}{p} <\frac{p-1}{p}.
\]
So $\left| 1- \left( \chi, \omega_0^{p^{r_{k_q}-1}}\right)\right| \geqslant  \frac{\pi}{p} >\varepsilon$. This inequality contradicts to (\ref{5}). Now we can choose $k'_0 >k_0$ such that $n_{k-1} < n_k - C_1, \forall k> k'_0$.

2) Set
\[
T(k) =\{ i: \; n_{k-1} +1<i< n_k +1 \mbox{ such that } b_i =p-1 \mbox{ and }b_{i+1}=0 \}.
\]
If $T(k)$ is not empty, we set $t_k =\min \{ i:\; i\in T(k)\}$.

{\it Let us prove that there exists $C_2 >(C_1)>0$  such that for every $k>k'_0$ if $i\in T(k)$, then $0\leqslant n_k -i <C_2$. Hence there exists $k''_0 >k'_0$ such that for every $k>k''_0$ and $i\in T(k)$ if $b_i >0$, then  $b_i = b_{i+1}=\dots = b_{n_k -C_2} =p-1$.}

Assume the converse and there exists a subsequence $t_{k_q}$ such that $n_{k_q} -t_{k_q} \to\infty$. We can assume that $k_1 >k_0$ and $n_{k_q} -r_{k_q} >r_0$.  Since
\[
\omega_0^{p^{t_{k_q}-1}} =(0,\dots,0_{n_{k_q -1}},\dots, 0, 1_{t_{k_q} },0,\dots)
\]
 satisfies condition (i), then it is contained in $U_{\delta}$. On the other hand,
\[
\left( \chi, \omega_0^{p^{t_{k_q} -1}}\right) =\exp\left\{ 2\pi i \left( \frac{p-1}{p} +\sum_{i=1}^\infty \frac{b_{{t_{k_q}} +i}}{p^{i+2}} \right)\right\} \mbox{ and }
\]
\[
1-\frac{1}{p}\leqslant \frac{p-1}{p}  +\sum_{i=1}^\infty \frac{b_{{t_{k_q}} +i}}{p^{i+2}} \leqslant \frac{p-1}{p} + \frac{p-1}{p^3} \frac{1}{1-1/p} =1- \frac{p-1}{p^2}.
\]
So $\left| 1- \left( \chi, \omega_0^{p^{t_{k_q} -1}}\right)\right| \geqslant  \frac{\pi}{p^2} >\varepsilon$. This inequality contradicts to (\ref{5}). Choose $k''_0 >k'_0$ such that $n_{k-1} < n_k - C_2, \forall k> k''_0$.

3) {\it Let us prove that there exist $C_3 >C_2$ and $k'''_0 >k''_0$ such that for every $k>k'''_0$ either $b_{n_{k-1} +2} =b_{n_{k-1} +3}=\dots =b_{n_{k} -C_3}=p-1$ or $b_{n_{k-1} +2} =b_{n_{k-1} +3}=\dots =b_{n_{k} -C_3}=0$.}

Set $S(k) = \{ i: \; n_{k-1} +1<i< n_k -C_2 \mbox{ and } b_i =p-1  \}$. Denote by $s_k =\min \{ i\in S(k)\}$ if $S(k)\not=\emptyset$ and $s_k = n_k -C_2$ otherwise.
By item 2), it is enough to prove that the sequence
$\{ n_k - s_k, \mbox{ where } k \mbox{ is chosen such that } s_k >n_{k-1} +2 \}$
is bounded (then we can put $C_3$ is the maximum of this sequence and choose $k'''_0 >k''_0$ such that $n_{k-1} +2 < n_k - C_3, \forall k> k'''_0$).

Assume the converse and there exists a subsequence $s_{k_q}$ such that $s_{k_q} >n_{k_q -1} +2$ and $n_{k_q} -s_{k_q} \to\infty$. Then $b_i=0$ for $n_{k_q -1} +1<i<s_{k_q}$ and
\[
\omega_0^{p^{s_{k_q}-2}} =(0,\dots,0_{n_{k-1}},\dots, 0, 1_{s_{k_q} -1},0,\dots)\in U_{\delta} \mbox{ for enough big } q.
\]
 On the other hand,
\[
\left( \chi, \omega_0^{p^{s_{k_q}-2}}\right) =\exp\left\{ 2\pi i \left( \frac{p-1}{p^2} +\sum_{i=1}^\infty \frac{b_{{s_{k_q}} +i}}{p^{i+2}} \right)\right\} \mbox{ and }
\]
\[
\frac{p-1}{p^2}\leqslant \frac{p-1}{p^2}  +\sum_{i=1}^\infty \frac{b_{{s_{k_q}} +i}}{p^{i+2}} \leqslant \frac{p-1}{p^2} + \frac{p-1}{p^3} \frac{1}{1-1/p} =\frac{1}{p}.
\]
So $\left| 1- \left( \chi, \omega_0^{p^{s_{k_q}-2}}\right)\right| \geqslant  \frac{\pi (p-1)}{p^2} >\varepsilon$. This inequality contradicts to (\ref{5}).

4) Set $A=\{ k:\; k>k'''_0 \mbox{ and } b_{n_{k-1} +2} =b_{n_{k-1} +3}=\dots =b_{n_{k} -C_3}=0 \}$. We can assume that $A$ is infinite. Indeed, if $A$ is finite, then
\[
(-\chi , \omega_0 ) =\exp \left\{ 2\pi i \left( \sum_{i=1}^\infty \frac{p-1-b_i}{p^i} \right)\right\}
\]
and we can consider the character $-\chi$ instead of $\chi$.

Denote by $l(k)=\min \{ i:\; \mbox{ where } n_k -C_3 <i \mbox{ and } b_i >0 \}, k>k'''_0$.

Assume that $\alpha\not\in \mathbb{Z} (p^\infty )$. Then, by item 3), there exist $a\geqslant  0$ and a subsequence $k_q$ such that $n_{k_q} +2 - l(k_q) =a$.
Choose $\widetilde{k}> k'''_0$ such that $n_k -n_{k-1}> r_0 +a+3$ for every $k>\widetilde{k}$. Set $h(q) = n_{k_q} -(r_0 +a+3)$. Then $l(k_q) -h(q) =r_0 +5$ and
\[
\omega_0^{p^{ h(q)-1}} =(0,\dots,0_{n_{k_q}-1},\dots, 0, 1_{h(q)} ,0,\dots)
\]
 satisfies condition (i).
Put $w(j)=(h(1)-1) + (h(2)-1)+\dots + (h(j)-1)$. Then $\omega_0^{p^{w(j)}}$ also satisfies condition (i) for every $j$ and, hence, it is contained in $U_{\delta}$. Since
\[
\left( \chi , \omega_0^{p^{ h(q)-1}}\right) =\exp \left\{ 2\pi i \left( \frac{b_{l_{k_q}}}{p^{r_0 +6}} +\sum_{i=1}^\infty \frac{b_i}{p^{r_0 +6+i}} \right)\right\} \mbox{ and}
\]
\[
\frac{b_{l_{k_q}}}{p^{r_0 +6}} <\frac{b_{l_{k_q}}}{p^{r_0 +6}} +\sum_{i=1}^\infty \frac{b_i}{p^{r_0 +6+i}} <\frac{b_{l_{k_q}}+1}{p^{r_0 +6}},
\]
then $({\rm mod}\; 1)$
\[
\frac{1}{p^{r_0 +6}} \sum_{q=1}^j b_{l_{k_q}} \leqslant \frac{1}{2\pi i} {\rm Arg} \left( \chi, \omega_0^{p^{w(j)}}\right) \leqslant \frac{1}{p^{r_0 +6}} \sum_{q=1}^j (b_{l_{k_q}}+1).
\]
It is clear that there exists $j$ such that $\left| 1- \left( \chi, \omega_0^{p^{m(j)}}\right)\right| > \varepsilon$. This inequality contradicts to (\ref{5}). Thus $\alpha \in \mathbb{Z} (p^\infty )$.
\end{proof}

{\it Let us prove Theorem \ref{t1} for the case $\mathbb{Z} (p^{\infty})$}.
By Lemma \ref{l13}, $G^\wedge_\mathbf{u} = \mathbb{Z} (p^\infty)$ algebraically. By Proposition 1 \cite{Ga1}, $G_\mathbf{u}$ is reflexive. So $\mathbb{Z} (p^\infty)$ with the topology of $G^\wedge_\mathbf{u}$ is also reflexive.

\end{proof}

\subsection{The case $G=\bigoplus_n \mathbb{Z} (b_n)$}

\begin{proof} Assume that $G=\bigoplus_n \mathbb{Z} (b_n)$, where $b_1 \leqslant b_2 \leqslant\dots, b_n \to\infty$.

The metric $d$ on $G_d^\wedge =\prod_n \mathbb{Z} (b_n)$ is defined as follows: if $\omega_1 \not= \omega_2 \in G_d^\wedge$, then $d(\omega_1 ,\omega_2) = 2^{-n}$, where  $n$ is the minimal index such that $a_n^1 \not= a_n^2$.

Set $\mathbf{u}=\{ u_n \}$, where $u_n =e_n$ is a generator of $\mathbb{Z} (b_n)$. Then
\[
s_\mathbf{u} (G_d^\wedge) = \left\{ \omega =(a_n) \in G_d^\wedge :\, (u_n , \omega)=\exp \left\{ 2\pi i \frac{a_n}{b_n}\right\} \to 1 \right\}.
\]
So: $\omega\in s_\mathbf{u} (G_d^\wedge)$ if and only if $\left\| \frac{a_n}{b_n} \right\| \to 0$.

Evidently, $G$ is dense in $G_d^\wedge $ and $G\subset s_{\mathbf{u}} (G_d^\wedge)$.
By Theorem 3 \cite{Ga1}, $\mathbf{u}$ is a $T$-sequence and $(G, \mathbf{u})^\wedge = G_\mathbf{u}$, where the Polish group $G_\mathbf{u}$ is $s_{\mathbf{u}} (G_d^\wedge)$ with the Polish group metric $\rho$ [see (\ref{01})].

{\bf 1.} {\it Let us prove that $G$ is dense in $G_\mathbf{u}$.}

Let $\omega =(a_n) \in G_\mathbf{u}$ and $\varepsilon >0$. Choose $n_0$ such that $|1-(u_n,\omega)|<\varepsilon/10$ for all $n\geqslant  n_0$. Choose $m\geqslant  n_0$ such that $d(\omega, \omega_m) <\varepsilon/10$, where $\omega_m = (a_1,\dots, a_m,0,\dots)$. Then $|(u_n, \omega)-(u_n,\omega_m)|=1$ for $n<n_0$ and $\rho (\omega,\omega_m)<\varepsilon/10 +\varepsilon/10 <\varepsilon$. Thus $G$ is dense in $G_\mathbf{u}$.

{\bf 2.} {\it Let us prove that $G_\mathbf{u}^\wedge =G$ algebraically}.

By item 1, the conjugate homomorphism $G_\mathbf{u}^\wedge \to G_d^\wedge$ is a monomorphism.
So any $\chi \in G_\mathbf{u}^\wedge$ we can represent in the form $\chi =(c_n) \in G_d^\wedge, 0\leqslant c_n <b_n$. We need to prove only that $c_n =0$ for all enough big $n$.

Let $\varepsilon >0$. Since $\chi$ is continuous, there exists an integer $M>10$ such that $|1-(\chi,\omega)|<\varepsilon, \forall \omega\in U_{1/M}$. By the definition of the metric $\rho$ on $G_\mathbf{u}$ (\ref{01}),  there exists $k_0$ such that if $\omega\in G_\mathbf{u}$ has the form $\omega = (0,\dots, 0_{n_{k_0} -1}, a_{n_{k_0}}, a_{n_{k_0}+1}, \dots)$ and
\begin{equation} \label{21}
\left| 1- \exp \left\{ 2\pi i \frac{a_n}{b_n} \right\} \right| <\frac{1}{M} \mbox{ for all } n\geqslant  n_{k_0}, \mbox{ then } \omega\in U_{1/M} ,
\end{equation}
and, in particular, $|1- (\chi,\omega)| <\varepsilon$.

Now assume the converse and there exists a sequence $n_1 < n_2<\dots$ such that $c_{n_k} >0$. We can assume that $\frac{c_{n_k}}{b_{n_k}}$ converges to $\lambda \in [0,1]$. There exist three possibilities.

1) $\lambda\in (0,1)$. Then we can assume that $\alpha < \left\| \frac{c_{n_k}}{b_{n_k}}\right\| \leqslant \frac{1}{2}$ for some $\alpha > 0$ and all $k$. Let $\varepsilon <\alpha$. Choose $k\geqslant  k_0$ such that $b_{n_k} >10M$ and set
\[
\omega = (0,\dots,0, 1_{n_k}, 0,\dots).
\]
Since $\left| 1- \exp \left\{ 2\pi i \frac{1}{b_{n_k}} \right\} \right| <\frac{2\pi}{b_{n_k}} <\frac{1}{M}$, then, by (\ref{21}), we have $\omega\in U_{1/M}$. Thus $|1- (\chi,\omega)| <\varepsilon$. On the other hand
\[
|1-(\chi, \omega) |=\left| 1-\exp\left\{ 2\pi i  \frac{c_{n_k}}{b_{n_k}}\right\} \right| \geqslant  \pi \left\| \frac{c_{n_k}}{b_{n_k}} \right\| >\pi\alpha >\varepsilon.
\]
It is a contradiction.

2) $\lambda =0$. Let $\varepsilon <0.01$. Choose $k\geqslant  k_0$ such that
\begin{equation} \label{22}
\frac{c_{n_l}}{b_{n_l}} < \frac{1}{20\pi M^3}, \mbox{ for every } l\geqslant  k.
\end{equation}
Set $a_{n_l} = \left[ \frac{b_{n_l} }{20\pi M c_{n_l}} \right]$ and $\varepsilon_l = \frac{b_{n_l} }{20\pi M c_{n_l}} -a_{n_l} <1$. Then, by (\ref{22}), $a_{n_l}  >M >0$. Put
\[
\omega =(0,\dots,0, a_{n_{k}},0,\dots,0, a_{n_{k+1}},0,\dots, 0,
a_{n_{k +M-1}}, 0,\dots).
\]
It is clear, by (\ref{21}), that $\omega\in U_{1/M}$ and, hence, $|1- (\chi,\omega)| <\varepsilon$. On the other hand, since
\[
\sum_{l=k}^{k+M-1} a_{n_{l}} \cdot \frac{c_{n_{l}}}{b_{n_{l}}} <\sum_{l=k}^{k+M-1} \frac{b_{n_l} }{20\pi M c_{n_l}} \cdot \frac{c_{n_{l}}}{b_{n_{l}}} =M\cdot \frac{1}{20\pi M} =\frac{1}{20\pi},
\]
and
\[
\sum_{l=k}^{k+M-1} a_{n_{l}} \cdot \frac{c_{n_{l}}}{b_{n_{l}}} = \sum_{l=k}^{k+M-1} \left( \frac{b_{n_l} }{20\pi M c_{n_l}} -\varepsilon_l \right) \cdot \frac{c_{n_{l}}}{b_{n_{l}}} =
M\cdot \frac{1}{20\pi M} - \sum_{l=k}^{k+M-1} \varepsilon_l  \cdot \frac{c_{n_{l}}}{b_{n_{l}}}>
\]
\[
(\mbox{by } (\ref{22})) > \frac{1}{20\pi} - \frac{M}{20\pi M^3} =\frac{1}{20\pi} \left( 1-\frac{1}{M^2} \right) > \frac{0.9}{20\pi},
\]
then
\[
|1-(\chi,\omega)| =\left| 1-\exp \left\{ 2\pi i \sum_{l=k}^{k+M-1} a_{n_{l}} \cdot \frac{c_{n_{l}}}{b_{n_{l}}} \right\}\right| >0.04 >\varepsilon .
\]
It is a contradiction.

3) $\lambda =1$. Let $\varepsilon <0.01$. Choose $k\geqslant  k_0$ such that
\begin{equation} \label{23}
\frac{b_{n_l}- c_{n_l}}{b_{n_l}} < \frac{1}{20\pi M^3}, \mbox{ for every } l\geqslant  k.
\end{equation}

Set $a_{n_l} = \left[ \frac{b_{n_l} }{20\pi M (b_{n_l}-c_{n_l})} \right]$ and $\varepsilon_l = \frac{b_{n_l} }{20\pi M (b_{n_l}- c_{n_l})} -a_{n_l} <1$. Then, by (\ref{23}), $a_{n_l} >M >0$. Put
\[
\omega =(0,\dots,0, a_{n_{k}},0,\dots,0, a_{n_{k+1}},0,\dots, 0,
a_{n_{k +M-1}}, 0,\dots).
\]
It is clear, by (\ref{21}), that $\omega\in U_{1/M}$ and, hence, $|1- (\chi,\omega)| <\varepsilon$. On the other hand, since
\[
\sum_{l=k}^{k+M-1} a_{n_{l}} \cdot \frac{c_{n_{l}}}{b_{n_{l}}} ({\rm mod} \, 1 ) = -\sum_{l=k}^{k+M-1} a_{n_{l}} \cdot \frac{b_{n_{l}} - c_{n_{l}}}{b_{n_{l}}},
\]
we can repeat the computations in item 2), and obtain that $\varepsilon <0.04< |1-(\chi,\omega)|$.
It is a contradiction. So $(G,\mathbf{u})^{\wedge\wedge} =G_\mathbf{u}^\wedge =G$ algebraically.

{\bf 3.} {\it Let us prove Theorem \ref{t1} for the case $G=\bigoplus_n \mathbb{Z} (b_n)$}.
By item {\bf 2}, $G^\wedge_\mathbf{u} = G$ algebraically. By Proposition 1 \cite{Ga1}, $G_\mathbf{u}$ is reflexive. So $G$ with the topology of $G^\wedge_\mathbf{u}$ is also reflexive.

\end{proof}

\bibliographystyle{amsplain}

\end{document}